\title{Admissible sets do not exist for all parameters}
\author{
Luke Pebody
}
\date{\today}

\documentclass[11pt]{article}
\usepackage{amssymb}
\usepackage{amsmath}
\usepackage{amsthm}
\usepackage{bbm}
\usepackage{mathtools}
\usepackage[margin=1in]{geometry}
\usepackage[backend=bibtex]{biblatex}
\newcommand{\F}{\mathbb{F}}

\newcommand{\supp}{\mathrm{Supp}\textrm{ }}

\newtheorem{thm}{Theorem}
\newtheorem{lem}[thm]{Lemma}
\newtheorem{clm}[thm]{Claim}
\newtheorem{cor}[thm]{Corollary}
\newtheorem{qn}[thm]{Question}

\theoremstyle{definition}

\begin{document}
\maketitle

\begin{abstract}
A cap set in $\F_3^n$ is a subset that contains no three elements adding to 0. 
Building on a construction of Edel~\cite{Edel}, a recent paper of Tyrrell~\cite{Tyrrell} gave
the first improvement to the lower bound for a size of a cap set in two decades
showing that, for large enough $n$, there is always a cap set in $\F_3^n$ of size
at least $2.218^n$. This was shown by constructing what is called an $I(11,7)$ admissible 
set.

An admissible set is a subset of $\{0,1,2\}^m$ such that the supports of the vectors form
an antichain with respect to inclusion and each triple of vectors has some coordinate where
either exactly one of them is non-zero or exactly two are and they have different values. 
Such an admissible set is said to be $I(m,w)$ if it is of size $\binom mw$ and all of the 
vectors have exactly $w$ non-zero elements. In Tyrrell's paper they conjectured that 
$I(m,w)$ admissible set exists for all parameters. We resolve this conjecture by showing that 
there exists an $N$ such that an $I(N,4)$ admissible set does not exist.

We refer to the type of a vector in $\{0,1,2\}^m$ is the ordered sequence of its non-zero
coefficients. The vectors of type $12$ form an $I(m,2)$ admissible set and the vectors of type
$121$ form an $I(m,3)$ admissible set (as can be easily checked by an interested reader). 
Sadly it is quite easily proved that there is no $I(6,4)$ admissible set where all
vectors are of the same type. It follows by Ramsey's Theorem applied to 4-regular hypergraphs
that there exists an $N$ such that an $I(N,4)$ admissible set does not exist.

A similar argument shows that there exists an $N'$ such that an $I(N',N'-2)$ admissible set
does not exist.
Since we can construct an $I(m-1,w)$ and an $I(m-1,w-1)$ admissible set from an
$I(m,w)$ admissible set, it follows that there are only finitely many $I(m,w)$ admissible sets exist other
than the known forms $I(m,1)$, $I(m,2)$, $I(m,3)$ and $I(m,m-1)$.
\end{abstract}

\section{Definitions}
Given a vector $v\in\{0, 1, 2\}^m$, the {\em support of v} (written $\supp v$) is 
$\{i:i\in[m], v_i\ne 0\}$. Given $S\subseteq[m]$, write $V_S$ for the set of vectors in
$\{0, 1, 2\}^m$ with support $S$. Clearly $|V_S|=2^{|S|}$. We 
say that a subset of $\{0, 1, 2\}^m$ is $I(m,w)$ if it contains exactly one element
of $V_S$ for all subsets $S$ of $[m]$ of size $w$ (and no other sets).

If vectors $v_1, v_2\in\{0,1,2\}^m$ have $\supp v_1\subseteq\supp v_2$,
we say that $\{v_1, v_2\}$ form a {\em clash}. 
Further, if vectors $v_1, v_2, v_3\in\{0, 1, 2\}^m$ are such that there is no coordinate $i$
for which exactly one of the vectors has a non-0 $i$ coordinate and also no coordinate $i$ for
which all of the $i$ coordinates are different, we say that $\{v_1, v_2, v_3\}$ form a 
{\em clash}. 
We say that a subset of $\{0, 1, 2\}^m$ is {\em admissible} if it does not contain a clash.

We note that by construction the supports of the vectors in an $I(m,w)$ set are an antichain,
so no two elements can form a clash, and so an $I(m,w)$ set is admissible precisely if it contains
no 3-element clash.

Tyrrell showed~\cite{Tyrrell} how to construct $I(11,6), I(11,7)$ and $I(10,6)$ admissible sets
and conjectured that $I(m,w)$ admissible sets exist for all $0<w<m$. We show in this note that 
conjecture is false, proving specifically
\begin{thm}\label{T:main}
There exist only finitely many $4\le w\le m-2$ for which $I(m,w)$ admissible sets exist.
\end{thm}

Since it is known that $I(m,1)$, $I(m,2)$, $I(m,3)$, $I(m,m-1)$ and $I(m,m)$ admissible sets 
exist (details are given in Section~\ref{S:closing}), this resolves the existence of
$I(m,w)$ admissible sets for all but finitely many pairs $(m, w)$.

\section{Types}
Given a vector $v\in\{0,1,2\}^n$ and vector $t\in\{1,2\}^k$, we say that {\em $v$ is of type 
$t$} if there are at least $k$ non-zero coordinates in $v$ and the 
$i^\textrm{th}$ non-zero coordinate of 
$v$ is the $i^\textrm{th}$ coordinate of $w$.

Given a vector $t\in\{1,2\}^k$, and subsets $S_1, S_2, S_3$ of $[m]$ of size $w$, say that
$S_1, S_2, S_3$ form a {\em type $t$ clash} if whenever $v_1, v_2, v_3$ are elements of
$V_{S_1}, V_{S_2}$ and $V_{S_3}$ respectively and are of type $t$, $\{v_1, v_2, v_3\}$ are a 
clash. As a reminder this means that there is no coordinate where exactly one
of the vectors is non-zero, and no coordinate where two are non-zero and they are different.

\begin{clm}\label{C:clashes matter}
If $S_1, S_2, S_3\in\binom{[m]}w$ form a type $t$ clash then there is no $I(m,w)$ admissible set
where all the elements are of type $t$.	
\end{clm}

\begin{proof}
Any $I(m,w)$ admissible set contains elements of $V_{S_1}, V_{S_2}$ and $V_{S_3}$. If all three
elements were of type $t$, they would form a clash and hence the set would not be admissible.	
\end{proof}

Now we show the existence of various typed clashes.

\begin{thm}\label{T:clashes exist}
$\{0134, 0234, 1234\}$ form a type $11$ clash, $\{0124, 0134, 0234\}$ form a type
$211$ clash, $\{0123, 0124, 0134\}$ form a type $1211$ clash and 
$\{0123, 0145, 2345\}$ form a type $1212$ clash.
\end{thm}

\begin{proof}
Let $a, b, c$ be elements of $V_{0134}$, $V_{0234}$ and $V_{1234}$ respectively of type 11.
Then for $0\le i\le 2$, $\{a_i, b_i, c_i\}$ is equal to $\{0, 1, 1\}$ and for $3\le i\le 4$,
all three elements are non-zero. Thus $\{a, b, c\}$ form a clash.

Similarly if $a, b, c$ are elements of $V_{0124}$, $V_{0134}$ and $V_{0234}$ respectively of 
type 211, then for $1\le i\le 3$, $\{a_i, b_i, c_i\}$ is equal to $\{0, 1, 1\}$ and for 
$i=0, 4$ all three elements are non-zero. Thus $\{a, b, c\}$ form a clash.

Similarly if $a, b, c$ are elements of $V_{0123}$, $V_{0124}$ and $V_{0214}$ respectively of 
type 1211, then for $2\le i\le 4$, $\{a_i, b_i, c_i\}$ is equal to $\{0, 1, 1\}$ and for 
$0\le i\le 1$ all three elements are non-zero. Thus $\{a, b, c\}$ form a clash.

Finally if $a, b, c$ are elements of $V_{0123}$, $V_{0145}$ and $V_{2345}$ respectively then
$a, b, c$ are $121200$, $120012$ and $001212$ which clearly form a clash.
\end{proof}

Given a vector $v\in\{0,1,2\}^m$, write $v^*$ for the vector for which 
$v^*_i=
\begin{cases}
0&\textrm{ if }v_i=0\\
3-v_i&\textrm{ otherwise}.
\end{cases}$

\begin{lem}\label{L:clashes can be starred}
If $S_1, S_2, S_3\in\binom{[m]}w$ form a type $t$ clash they also form a type $t^*$ clash.	
\end{lem}

\begin{proof}
If $v_1, v_2$ and $v_3$ are elements of $V_{S_1}$, $V_{S_2}$ and $V_{S_3}$ respectively of
type $t^*$, then $v_1^*$, $v_2^*$ and $v_3^*$ are elements of the same sets of type $t$ and so
by definition, they form a clash.

This means that for each coordinate, either none of the vectors are non-0, all of them are,
or there are exactly two which are non-0 and they are equal. The same is clearly true of
$v_1, v_2$ and $v_3$ so they form a clash.	
\end{proof}

This allows us to show that there are no small admissible sets all of the same type.

\begin{cor}\label{C:no monotype sets}
	For any $t$ in $\{11, 1211, 122, 211, 2122, 22\}$ there are no $I(5,4)$ admissible sets
	where each element is of type $t$. Further, for any $t$ in $\{1212, 2121\}$ there are no
	$I(6,4)$ admissible sets where each element is of type $t$.	
\end{cor}

\begin{proof}
	Theorem~\ref{T:clashes exist} shows the existence of type $t$ clashes in
	$\binom{[5]}4$ for $t\in\{11, 1211, 211\}$, from which Lemma~\ref{L:clashes can be starred}
	gives the same clashes for $t\in\{122, 2122, 22\}$, and then Claim~\ref{C:clashes matter}
	shows that there are no $I(5,4)$ admissible sets where each element is of type $t$ for all
	such $t$.
	
	Similarly Theorem~\ref{T:clashes exist} shows the existence of a type 1212 clash in
	$\binom{[6]}4$, from which Lemma~\ref{L:clashes can be starred} gives the same clash for
	type 2121 and then Claim~\ref{C:clashes matter} shows that there are no $I(6,4)$ admissible
	sets where each element is of type 1212 or where each element is of type 2121.
\end{proof}

This then allows us to show a specific $I(m,w)$ admissible set does not exist. Let 
$R^{(k)}(n_1, \ldots, n_c)$ denote the Ramsey number for $n_1, \ldots, n_c$ (ie the smallest number
$N$ such that if we colour the elements of $\binom{[N]}{k}$ with $c$ colours, there will exist
a $1\le i\le c$ such that some subset of $[N]$ of size $n_i$ is all coloured with colour $i$.)

\begin{cor}\label{C:boom}
	If $N\ge R^{(4)}(5, 5, 5, 5, 5, 5, 6, 6)$ then there does not exist an $I(N, 4)$ admissible set.	
\end{cor}

\begin{proof}
	Suppose otherwise. Consider such an $I(N,4)$ admissible set
	$\{v_S:S\in\binom{[N]}4\}$, where 
	$N=R^{(4)}(5, 5, 5, 5, 5, 5, 6, 6)$. 
	Firstly we note that all vectors in $\{0, 1, 2\}^n$ with at least 4 elements in their support
	are of precisely one of the types $\{11, 1211, 122, 211, 2122, 22, 1212, 2121\}$.
	
	Thus we can colour $\binom{[N]}4$ by colouring $S$ with whichever of these types $v_S$ is. 
	Then we will either get a subset $S$ of $[N]$ of size 5 for which all elements of
	$\binom{S}4$ are of the same type in $\{11, 1211, 122, 211, 2122, 22\}$ or a subset
	$S$ of $[N]$ of size 6 for which all elements of $\binom{S}4$ are of the same type in
	$\{1212, 2121\}$. Either way Corollary~\ref{C:no monotype sets} means there is a clash in 
	our supposedly admissible set, which forms a contradiction.
\end{proof}

\section{$I(m,m-2)$ sets}
Suppose we have an $I(m,m-2)$ admissible set $S$. For all $1\le i<j\le n$, there is exactly one vector 
$v(i,j)$ in $S$ such that $v(i,j)_i=0$ and $v(i,j)_j=0$. 
Let us induce an 8-colouring $c_S$ on the elements of
$\binom{[m]}3$ by $c_S((i,j,k))=(v(i,j)_k, v(i,k)_j, v(j,k)_i)$ for $i<j<k$.

\begin{lem}\label{L:triples}
	If all triples have the same colour from $\{111, 112, 122, 211, 221, 222\}$ then $m\le 3$. 
	If all triples have the same colour from $\{121, 212\}$ then $m\le 5$.
\end{lem}

\begin{proof}
	Note if we know the colour of all triples, we know the entire set. If $m\ge 4$ and all triples
	have colour 111 or 112, then the vectors $v(1,4)$, $v(2,4)$ and $v(3,4)$ are 
	$0110\ldots$, $1010\ldots$ and $1100\ldots$
	which clash. By inverting the colours, we get the same result for 221 and 222. Further, 
	by reversing the strings which clash for 112, we get strings which clash for 211 (and then by
	inverting the colours for 122).
	
	If all triples have colour 121 and $m\ge 6$, then the vectors $v(1,2)$, $v(3,4)$ and $v(5,6)$ are
	$001111\ldots$, $110011\ldots$ and $111100\ldots$ which clash. Inverting the colours gives the
	same result for 212.
\end{proof}

Then as before Ramsey's Theorem shows that $I(m,m-2)$ admissible sets only exist for finitely many $m$.

\begin{cor}\label{C:mic}
	If $m\ge R^{(3)}(4,4,4,4,4,4,6,6)$ then an $I(m,m-2)$ admissible set does not exist.
\end{cor}

\begin{proof}
	If we have such an admissible set $S$ then the colouring $c_S$ would either have a monochromatic
	set of size at least 4 in a colour from $\{111, 112, 122, 211, 221, 222\}$ or a monochromatic
	set of size at least 6 in a colour from $\{121, 212\}$. Lemma~\ref{L:triples} shows no such
	monochromatic set can occur.
\end{proof}

\section{Summary}\label{S:closing}
We can then prove Theorem~\ref{T:main} by combining Corollary~\ref{C:boom} and Corollary~\ref{C:mic}.
\begin{proof}[Proof of Theorem~\ref{T:main}]
	Suppose $S$ is an $I(n, w)$ admissible set $S$ where $n>w>0$ and let 
	$S_1=\{(v_1, \ldots, v_{n-1}):(v_1, \ldots, v_n)\in S, v_n=0\}$ and
	$S_2=\{(v_1, \ldots, v_{n-1}):(v_1, \ldots, v_n)\in S, v_n\ne0\}$. 
	Then it is clear that $S_1$ is an $I(n-1, w)$ admissible set and $S_2$ is
	an $I(n-1, w-1)$ admissible set.

	Thus if no $I(m, w)$ admissible set exists then nor does an $I(m+1, w)$ admissible set
	or an $I(m+1, w+1)$ admissible set.
	
	By Corollary~\ref{C:boom} there does not exist an $I(N, 4)$ admissible set for 
	\[N>=C_1=R^{(4)}(5, 5, 5, 5, 5, 5, 6, 6).\] It then follows that if an $I(m, w)$ admissible set
	exists with $4\le w\le m-2$, then an $I((m-w)-4, 4)$ admissible set exists and so
	$m-w<C_1+4$.
	
	By Corollary~\ref{C:mic} there does not exist an $I(N, N-2)$ admissible set for
	\[N>=C_2=R^{(3)}(4, 4, 4, 4, 4, 4, 6, 6).\] It then follows that if an $I(m, w)$ admissible set
	exists with $4\le w\le m-2$, then an $I(w+2, w)$ admissible set exists and so 
	$w<C_2$.
	
	Thus we see that if an $I(m, w)$ admissible set
	exists with $4\le w\le m-2$ then both $w$ and $m-w$ are bounded, so only finitely many such
	can exist.
\end{proof}

For numbers in the range $m>w>0$ which do not satisfy $4\le w\le m-2$, there are known constructions
for $I(m,w)$ admissible sets. Namely the vectors of weight 1 and type 1 form an $I(m,1)$ admissible 
set, the vectors of weight 2 and type 12 form an $I(m,2)$ admissible set and the vectors of weight 3
and type 123 form an $I(m,3)$ admissible set. Finally for $m\ge 2$, the vectors 
\begin{align*}
v(i)&=(v(i)_1, v(i)_2, \ldots, v(i)_m):1\le i\le m	\\
v(i)_j&=\begin{cases}
	1\textrm{ if }i<j,\\
	2\textrm{ if }i>j,\\
	0\textrm{ if }i=j
\end{cases}
\end{align*}
form an $I(m,m-1)$ admissible set. Thus only the sets with $4\le w\le m-2$ were open and 
Theorem~\ref{T:main} resolves their existence for all but a finite (admittedly somewhat large)
set of values.

In~\cite{Tyrrell}, Tyrrell conjectured the existence of all $I(m, w)$ sets as if one could prove
that $I(31k, 28k)$ admissible sets exist, then for all $\epsilon>0$, it would follow that
for sufficiently large $n$ there exist cap sets in $\mathbb{F}_3^n$ of size at least
$(124^{1/6}-\epsilon)^n$. Unfortunately, we have now seen that these sets
do not exist for sufficiently large $k$. 

However, the main result in~\cite{Tyrrell} is a 
strong bound which is proved by constructing an admissible set with supports in 
$\binom{[1562]}{990}$ which is very large but does not contain all $\binom{1562}{990}$ possible 
elements. 

\begin{qn}
\begin{enumerate}
\item Which $I(m,w)$ admissible sets exist with $4\le w\le m-2$? From Tyrrell it is known that all such
	  exist for $m\le 11$ and we have also used the answer set programming library clingo to create an I(12,4)
	  admissible set.
\item For $m>w\ge 4$ what is the maximum number $f(m, w)$ of weight $w$ vectors one can have in 
	an admissible set in $\{0, 1, 2\}^m$?
\end{enumerate}
\end{qn}

Given an admissible set in $\{0, 1, 2\}^m$ with $f(m, w)$ weight $w$ vectors, one can 
construct a capset in $\F_3^{36m}$ of size $f(m, w)(72\times 112^5)^{m-w}(112^6)^w$ and hence 
can show that for all sufficiently large $n$, there exist cap-sets in $\F_3^{n}$ of size
$(f(m,w)^{1/36m}(72\times 112^5)^{1-w/m}(112^6)^{w/m}-\epsilon)^{1/n}$.

\printbibliography
\end{document}